\documentclass[12pt]{amsart}
\usepackage{amsmath,amssymb,amsbsy,amsfonts,latexsym,amsopn,amstext,cite,
                                               amsxtra,euscript,amscd,bm,mathabx}
\usepackage{url}
\usepackage[colorlinks,linkcolor=blue,anchorcolor=blue,citecolor=blue,backref=page]{hyperref}
\usepackage{color}
\usepackage{graphics,epsfig}
\usepackage{graphicx}
\usepackage{float} 
\usepackage[english]{babel} 
\usepackage{mathtools}
\usepackage{todonotes}
\usepackage{url}
\usepackage[colorlinks,linkcolor=blue,anchorcolor=blue,citecolor=blue,backref=page]{hyperref}

\usepackage[norefs,nocites]{refcheck}

\hypersetup{breaklinks=true}

\usepackage[norefs,nocites]{refcheck}
\usepackage[english]{babel}
\begin{document}

\newtheorem{thm}{Theorem}
\newtheorem{lem}[thm]{Lemma}
\newtheorem{claim}[thm]{Claim}
\newtheorem{cor}[thm]{Corollary}
\newtheorem{prop}[thm]{Proposition} 
\newtheorem{definition}[thm]{Definition}
\newtheorem{question}[thm]{Open Question}
\newtheorem{qn}[thm]{Question}
\newtheorem{conj}[thm]{Conjecture}
\newtheorem{prob}{Problem}

\theoremstyle{remark}
\newtheorem{rem}[thm]{Remark}

\newcommand{\GL}{\operatorname{GL}}
\newcommand{\SL}{\operatorname{SL}}
\newcommand{\lcm}{\operatorname{lcm}}
\newcommand{\ord}{\operatorname{ord}}
\newcommand{\Op}{\operatorname{Op}}
\newcommand{\Tr}{\operatorname{Tr}}
\newcommand{\Nm}{\operatorname{Nm}}

\numberwithin{equation}{section}
\numberwithin{thm}{section}
\numberwithin{table}{section}

\def\vol {{\mathrm{vol\,}}}
\def\squareforqed{\hbox{\rlap{$\sqcap$}$\sqcup$}}
\def\qed{\ifmmode\squareforqed\else{\unskip\nobreak\hfil
\penalty50\hskip1em\null\nobreak\hfil\squareforqed
\parfillskip=0pt\finalhyphendemerits=0\endgraf}\fi}

\def \balpha{\bm{\alpha}}
\def \bbeta{\bm{\beta}}
\def \bgamma{\bm{\gamma}}
\def \blambda{\bm{\lambda}}
\def \bchi{\bm{\chi}}
\def \bphi{\bm{\varphi}}
\def \bpsi{\bm{\psi}}
\def \bomega{\bm{\omega}}
\def \btheta{\bm{\vartheta}}
\def \ochi{\overline{\chi}}

\def\eps{\varepsilon}

\newcommand{\bfxi}{{\boldsymbol{\xi}}}
\newcommand{\bfrho}{{\boldsymbol{\rho}}}

\def\Kab{\sfK_\psi(a,b)}
\def\Kuv{\sfK_\psi(u,v)}
\def\SaUV{\cS_\psi(\balpha;\cU,\cV)}
\def\SaAV{\cS_\psi(\balpha;\cA,\cV)}

\def\SUV{\cS_\psi(\cU,\cV)}
\def\SAB{\cS_\psi(\cA,\cB)}

\def\Kmnp{\sfK_p(m,n)}

\def\KKap{\cH_p(a)}
\def\KKaq{\cH_q(a)}
\def\KKmnp{\cH_p(m,n)}
\def\KKmnq{\cH_q(m,n)}

\def\Klmnp{\sfK_p(\ell, m,n)}
\def\Klmnq{\sfK_q(\ell, m,n)}

\def \SALMNq {\cS_q(\balpha;\cL,\cI,\cJ)}
\def \SALMNp {\cS_p(\balpha;\cL,\cI,\cJ)}

\def \SACXMQX {\fS(\balpha,\bzeta, \bxi; M,Q,X)}

\def\SAMJp{\cS_p(\balpha;\cM,\cJ)}
\def\SAMJq{\cS_q(\balpha;\cM,\cJ)}
\def\SAqMJq{\cS_q(\balpha_q;\cM,\cJ)}
\def\SAJq{\cS_q(\balpha;\cJ)}
\def\SAqJq{\cS_q(\balpha_q;\cJ)}
\def\SAIJp{\cS_p(\balpha;\cI,\cJ)}
\def\SAIJq{\cS_q(\balpha;\cI,\cJ)}

\def\RIJp{\cR_p(\cI,\cJ)}
\def\RIJq{\cR_q(\cI,\cJ)}

\def\TWXJp{\cT_p(\bomega;\cX,\cJ)}
\def\TWXJq{\cT_q(\bomega;\cX,\cJ)}
\def\TWpXJp{\cT_p(\bomega_p;\cX,\cJ)}
\def\TWqXJq{\cT_q(\bomega_q;\cX,\cJ)}
\def\TWJq{\cT_q(\bomega;\cJ)}
\def\TWqJq{\cT_q(\bomega_q;\cJ)}

 \def \xbar{\overline x}
  \def \ybar{\overline y}

\def\cA{{\mathcal A}}
\def\cB{{\mathcal B}}
\def\cC{{\mathcal C}}
\def\cD{{\mathcal D}}
\def\cE{{\mathcal E}}
\def\cF{{\mathcal F}}
\def\cG{{\mathcal G}}
\def\cH{{\mathcal H}}
\def\cI{{\mathcal I}}
\def\cJ{{\mathcal J}}
\def\cK{{\mathcal K}}
\def\cL{{\mathcal L}}
\def\cM{{\mathcal M}}
\def\cN{{\mathcal N}}
\def\cO{{\mathcal O}}
\def\cP{{\mathcal P}}
\def\cQ{{\mathcal Q}}
\def\cR{{\mathcal R}}
\def\cS{{\mathcal S}}
\def\cT{{\mathcal T}}
\def\cU{{\mathcal U}}
\def\cV{{\mathcal V}}
\def\cW{{\mathcal W}}
\def\cX{{\mathcal X}}
\def\cY{{\mathcal Y}}
\def\cZ{{\mathcal Z}}
\def\Ker{{\mathrm{Ker}}}
\def\g{{\mathrm{gcd}}}

\def\NmQR{N(m;Q,R)}
\def\VmQR{\cV(m;Q,R)}

\def\Xm{\cX_m}

\def \A {{\mathbb A}}
\def \B {{\mathbb A}}
\def \C {{\mathbb C}}
\def \N {{\mathbb N}}
\def \F {{\mathbb F}}
\def \G {{\mathbb G}}
\def \L {{\mathbb L}}
\def \K {{\mathbb K}}
\def \PP {{\mathbb P}}
\def \Q {{\mathbb Q}}
\def \R {{\mathbb R}}
\def \Z {{\mathbb Z}}
\def \fS{\mathfrak S}

\def\e{{\mathbf{\,e}}}
\def\ep{{\mathbf{\,e}}_p}
\def\eq{{\mathbf{\,e}}_q}
\def\er{{\mathbf{\,e}}_R}
\def\esr{{\mathbf{\,e}}_r}
\def\\{\cr}
\def\({\left(}
\def\){\right)}
\def\fl#1{\left\lfloor#1\right\rfloor}
\def\rf#1{\left\lceil#1\right\rceil}

\def\Tr{{\mathrm{Tr}}}
\def\Nm{{\mathrm{Nm}}}
\def\Im{{\mathrm{Im}}}

\def \oF {\overline \F}

\newcommand{\pfrac}[2]{{\left(\frac{#1}{#2}\right)}}

\def \Prob{{\mathrm {}}}
\def\e{\mathbf{e}}
\def\ep{{\mathbf{\,e}}_p}
\def\epp{{\mathbf{\,e}}_{p^2}}
\def\em{{\mathbf{\,e}}_m}

\def\Res{\mathrm{Res}}
\def\Orb{\mathrm{Orb}}

\def\vec#1{\mathbf{#1}}
\def \va{\vec{a}}
\def \vb{\vec{b}}
\def \vm{\vec{m}}
\def \vu{\vec{u}}
\def \vv{\vec{v}}
\def \vx{\vec{x}}
\def \vy{\vec{y}}
\def \vz{\vec{z}}
\def\flp#1{{\left\langle#1\right\rangle}_p}
\def\T {\mathsf {T}}

\def\sfG {\mathsf {G}}
\def\sfK {\mathsf {K}}

\def\mand{\qquad\mbox{and}\qquad}

\title[Character sums with elliptic sequences]
{On the correlations between character sums of division polynomials under shifts}

\author{Subham Bhakta}
\address{School of Mathematics and Statistics, University of New South Wales, Sydney, NSW 2052, Australia.} 
\email{subham.bhakta@unsw.edu.au}

\author{Igor E. Shparlinski}
\address{School of Mathematics and Statistics, University of New South Wales, Sydney, NSW 2052, Australia.} 
\email{igor.shparlinski@unsw.edu.au}

\begin{abstract}
Let $E$ be an elliptic curve over the finite field $\mathbb{F}_p$, and $P \in E(\mathbb{F}_p)$ be an $\F_p$-rational point. We study the sums
\[
S_{\chi,P}(N,h) = \sum_{n=1}^N \chi(\psi_n(P)) \chi(\psi_{n+h}(P)),
\]
where $\psi_n(P)$ denotes the $n$-th division polynomial evaluated at $P$, and $\chi$ is a multiplicative character of $\F_p^{*}$. We estimate $S_{\chi,P}(N,h)$ on average over $h$ over 
a rather short interval $h \in [1, H]$. We also obtain a multidimensional generalisation of this result. 
\end{abstract} 


\maketitle
\tableofcontents

\section{Introduction}
\subsection{Set-up and motivation}
Let $ E $ be an elliptic curve over $ \F_p $ of characteristic $ p > 3 $ given by the Weierstrass equation as following, with coefficients in $\F_p$.
\[
y^2=x^3+ax+b.
\]
For any integer $n\geq0$, define the $n$th division polynomial $\psi_n\in\F_p[x,y]$ as follows.
\begin{align*}
\psi_0 &= 0, \quad \psi_1 = 1, \quad \psi_2 = 2y, \\
\psi_3 &= 3x^4 + 6a x^2 + 12b x - a^2, \\
\psi_4 &= 4y(x^6 + 5a x^4 + 20b x^3 - 5a^2 x^2 - 4abx - 8b^2 - a^3),
\end{align*}
for more details, see~\cite[Exercise~3.7]{Silv2}, combined with the discussion in~\cite[Chapter~III.1]{Silv2}.

Denote by $ E(\F_p) $ the group of points on $ E $ defined over $ \F_p $. We can interpret each $\psi_n$ as rational functions on $E(\F_p)$. The multiplication of $P \in E$ by $n$ is given as a rational map by 
\begin{align*}
[n](P) &=
\biggl(
\frac{x(P)\psi_n^2(P)  - \psi_{n-1}(P)\psi_{n+1}(P)}{\psi_n(P)^2}, \\
& \qquad \qquad \frac{\psi_{n-1}(P)^2\psi_{n+2}(P) - \psi_{n-2}(P)\psi_{n+1}(P)^2}{4y(P)\psi_n(P)^3}
\biggr).
\end{align*}
It is well known that
\begin{align*}
\psi_{m+n}(P)&\psi_{m-n}(P)\psi_r^2(P)\\
& =\psi_{m+r}(P)\psi_{m-r}(P)\psi_n^2(P)-\psi_{n+r}(P)\psi_{n-r}(P)\psi_m(P)^2,
\end{align*}
for any integers $m,n,r$, see, for example,~\cite[Exercise~3.7(g)]{Silv2}.
 It turns out that $\psi_n(P)$ is periodic, and the period $T$, which  could be as large as $(p-1) \ord P$, 
 where $\ord P$ is the order of $P$ in the group  $ E(\F_p) $, see~\cite[Corollary~9]{Silv1}.   
 
Let $\chi$ be a multiplicative character of $\F_p^{*}$. We consider the sequence $\chi(\psi_n(P))$, where of course we set $\chi(0)=0$. Then, $\chi(\psi_n(P))$ is periodic with a much smaller period (see Section~\ref{sec:results}), compared to the period of $\psi_n(P)$. For points $P\in E(\F_p)$ of large orders, Shparlinski and Stange~\cite{ShSt} have studied the character sums of the form 
\[
S_{\chi,P}(N) =  \sum_{n=1}^N \chi(\psi_n(P)),
\]
for the quadratic character of $\F_p$ and optained a nontrivial bound provided $\ord P \ge N \ge p^{1/2 +\varepsilon}$ for some fixed $\varepsilon > 0$. 

Here, we   study a related but more difficult question on the correlation of $\chi(\psi_n(P))$ with its shifts, that is, about bounding the sums
\[
S_{\chi,P}(N,h)= \sum_{n=1}^N\chi(\psi_n(P))\chi(\psi_{n+h}(P))
\]
with an integer $h\ne 0$. 
 More precisely, we ask the following: 

\begin{qn}\label{qn:chowlaelliptic}
Let $p$ be a prime, $P\in E(\mathbb{F}_p)$ be a point and $\chi$ be a multiplicative character of $\F_p^{*}$. When do we have
\[
S_{\chi,P}(N,h)=o(N)?
\]
\end{qn}

Due to the limitations of our techniques, we do not have an affirmative answer to the question above. However, we prove an average version. For instance, we estimate the following sum
\[
\sum_{h=1}^H|S_{\chi,P}(N,h)|=\sum_{h=1}^H\left| \sum_{n=1}^N\chi(\psi_n(P))\chi(\psi_{n+h}(P))\right|.
\]

Question~\ref{qn:chowlaelliptic} itself, as well as our results, are motivated by the progress towards the
Chowla conjecture for $ \lambda(n)$ and $\mu(n)$, that is, 
for the Liouville and M{\"o}bius functions, respectively.

It is well-known that the prime number theorem is equivalent to any of the following 
statements
\[ \sum_{n=1}^N \lambda(n)=o(N) \quad \text{and} \quad   \sum_{n=1}^N \mu(n)=o(N),\qquad \text{as}\ N \to \infty, 
\]
Chowla~\cite{Chow} has conjectured that for any integer $h>0$, the following holds
\begin{equation}\label{eqn:ChowlaConj}
 \sum_{n=1}^N \lambda(n)\lambda(n+h)=o(N) \quad \text{and} \quad  \sum_{n=1}^N \mu(n)\mu(n+h)=o(N), 
\end{equation}
as $N \to \infty$. 

Despite several striking results towards these conjectures, both remain widely open, 
see~\cite{MRT, MRTTZ,Tao, TaTe1, TaTe2} and references therein. An averaged version of this conjecture has been studied in~\cite{MRT}, in fact for a more general class of multiplicative functions. Our function $\chi\(\psi_n(P)\)$ is not multiplicative, although it satisfies an \textit{almost multiplicative} property, as stated in Lemma~\ref{lem:mult}.  Thus,  Question~\ref{qn:chowlaelliptic} essentially studies an elliptic analogue of Chowla's conjectures~\eqref{eqn:ChowlaConj}, for the function $\chi\(\psi_n(P)\)$, and also remains widely open. 

\subsection{Notation} Throughout the whole article, as usual, the notations $U = O(V)$, $U \ll V$, and $ V\gg U$ are equivalent to $|U|\le c V$ for some positive constant $c$, which may depend on the  the order 
$d$ of the multiplicative character  $\chi$ and on the dimension $m$.

We use $\omega(n)$ and $\tau(n)$ to denote the number of distinct prime and positive integer factors of an integer $n\ne 0$ and use $\varphi(n)$ to denote the Euler function.

We also denote $\er(n) = \exp(2 \pi i n/R)$. 

Everywhere, we use $p$ to denote a prime, and $h,k,m,n$ are reserved for integers. 

\subsection{Main result}\label{sec:results}
Let $\chi$ be a multiplicative character on $\F_p^{*}$ of order $d$, and $P\in E(\mathbb{F}_p)$ be a point. Denote $\ord P$ be the order of $P$ in $E(\mathbb{F}_p)$, and set
\begin{equation}\label{eqn:R}
R=d  \ord P. 
\end{equation}
Applying~\cite[Lemma~3.1]{ShSt} for $s=d$, the sequences $\chi(\psi_n(P))$ is periodic with a period dividing $R$, as long as $\ord P\geq 3$.  

Here, given $m$ integers $h_1, \ldots, h_m$, we consider the following average values of multidimensional correlations
between the values of  $\chi(\psi_n(P))$:
    \begin{align*}
     & U_{m, \chi, P} (H,N)  =  \frac{1}{H^{m-1}}\sum_{h_2, \ldots, h_m=1}^H \left| \sum_{n=1}^N \chi(\psi_n(P)) \prod_{j=2}^m \chi(\psi_{n+h_j}(P))\right|, \\
& V_{m, \chi, P} (H, N)  =  \frac{1}{H^m} \sum_{h_1, \ldots, h_m=1}^H \left| \sum_{n=1}^N \prod_{j=1}^m \chi(\psi_{n+h_j}(P))\right|. 
\end{align*}

\begin{thm}\label{thm:intshift} 
For any integer $m\geq 2$, and integers $1\leq H,N\leq R$, 
we have
\begin{equation}\label{eq:B1}
\begin{split}
U_{m, \chi, P} (H,N)&, V_{m, \chi, P} (H,N)\\
&  \ll   H^{-1/8}R\exp\(O\((\log R)^{1/2}/\log \log R\)\) \\
& \qquad \qquad \qquad \qquad \qquad \quad + H^{1/2}R^{3/4}p^{1/8} \log R, 
\end{split}
\end{equation}
and  if 
$ R\ge p^{1/2}\exp\(2.1\log p/\log \log p\)$, 
then we also have 
\begin{equation}\label{eq:B2}
\begin{split}
U_{m, \chi, P} (H,N)&, V_{m, \chi, P} (H,N)\\
&   \ll H^{-1/4} R^{7/6}p^{1/24} (\log R) (\log \log R)^{1/6} .
\end{split}
\end{equation}
 \end{thm} 

\begin{rem}\label{rem:compl sums}
It is easy to see from the proof of Theorem~\ref{thm:intshift}  that for complete sums $N = R$, one can remove $\log R$ from the second term of the bound~\eqref{eq:B1} and from 
 the bound~\eqref{eq:B2}.
\end{rem}

To understand the strength of  Theorem~\ref{thm:intshift} we consider the extreme case of very 
long sums with $R \ge N \gg R$. Then the bound~\eqref{eq:B1} is nontrivial in the range
\[
\exp\(C(\log R)^{1/2}/\log \log R\) \le H \le  R^{1/2}  p^{-1/4} (\log R)^{-2} f(p)^{-1}
 \]
for some absolute constant $C> 0$ and any function $f(p)\to 0$
as $p\to \infty$. In fact, for any fixed $ \varepsilon> 0$, Remark~\ref{rem:withgamma} allows us to have a non-trivial bound in the extended lower range
$$H \ge \exp\(C_{\varepsilon}(\log R)^\varepsilon\),  
$$ 
where $C_{\varepsilon}>0$ is a constant depending on $\varepsilon$.

On the other hand,  the bound~\eqref{eq:B2} is nontrivial only when $H$ is quite large, 
namely  $R \ge H \ge  R^{2/3} p^{1/6}  (\log \log R)^{2/3}F(p) $ with any function $F(p)\to \infty$
as $p\to \infty$. 

Unfortunately there is a wide gap between this ranges as 
$$
R^{2/3} p^{1/6} /  R^{1/2}  p^{-1/4} \ge p^{5/12} R^{1/6} \gg p^{1/2}.
$$
(since the we have $R \ge p^{1/2}$ for either  of the 
bounds~\eqref{eq:B1} and~\eqref{eq:B2}  to be nontrivial). Closing this gap is certainly an interesting open question.

Another interesting regime is $H=N$. In this case the  bound~\eqref{eq:B1} is always trivial,  while~\eqref{eq:B2} requires  
$$
R \ge H=N \ge R^{14/15}p^{1/30} (\log R)^{4/5} (\log \log R)^{2/15}F(p)
$$ 
with $F(p)\to \infty$ as $p\to \infty$.

In a response to Question~\ref{qn:chowlaelliptic}, we have the following simple application of Theorem~\ref{thm:intshift}.
Namely,  taking $m=2$ and applying~\eqref{eq:B2}  we immediately derive the following result.

\begin{cor}\label{cor:ansqn}
    Let $\Delta > 0$ be an arbitrary real number. If $ R\ge p^{1/2}\exp\(2.1\log p/\log \log p\)$ then, we have
\begin{align*}
\sharp\,\left\{ 1 \leq h \leq H :~\left| S_{\chi, P, h}(N) \right| \ge  \Delta N\right\}&\\
\ll   \Delta^{-1} H^{3/4} N^{-1} &  R^{7/6}p^{1/24} (\log \log R)^{1/6} .
    \end{align*}
\end{cor}

We note that, ignoring logarithmic terms, Corollary~\ref{cor:ansqn} produces a meaningful 
result if $H^{1/4} N \ge R^{7/6}p^{1/24+\varepsilon}$ for some fixed $\varepsilon> 0$.

\section{Preliminaries}\label{sec:elliptic}

\subsection{Some properties of   division polynomials} 

We have an \textit{almost multiplicative} nature of $\chi(\psi_n(P))$, which is given in~\cite[Lemma~3.2]{ShSt}. 

\begin{lem}\label{lem:mult}
    Let $\chi$ be a multiplicative character. Then for any integers $m,n$   we have
\[
\chi(\psi_{mn}(P))=\chi(\psi_m(nP))\chi(\psi_n(P))^{m^2}.
\]
\end{lem}

As usual, for any $\Psi\in \F_p(E)$, we define
\[
\deg \Psi=\sum_{\substack{P\in E(\overline{\F_p})\\ \nu_P(\Psi)>0}} \nu_{P}(\Psi),
\]
where  $\nu_{P}(\Psi)$ is the multiplicity of $P$ as a zero of $\Psi$ (and thus this is a finite sum over 
all zeros $P$ of $\Psi$). 

Moreover, we also need the following results to show that a certain class of functions satisfies the requirement of Lemma~\ref{lem:key} below. 

\begin{lem}\label{lem:notpower}
Let $m$ be any integer such that $\mathrm{ord}(P)\not\mid m$. For positive integers $k \neq \ell$ with $\gcd(k\ell ,R)=1$, we consider the function
\[
\Psi_P(Q)=\psi_{k}(Q)\cdot \psi_{k}^{-1}(Q+mP)\cdot \psi^{-1}_{\ell}(Q)\cdot \psi_{\ell}(Q+mP).
\]
Then,
\begin{enumerate}
\item[(i)] $\deg \Psi_P < 2(k^2+ \ell^2)$; 
    \item[(ii)] $\Psi_P$ is not a non-trivial power of any function in the function field $\overline{\F_p(E)}$.
    \end{enumerate}
\end{lem}

\begin{proof}
Proof of~(i) is clear, as each $\psi_h$ has $h^2-1$ simple zeroes,
see~\cite[Exercise~III.3.7]{Silv2}.  

To prove~(ii),  let us assume that $\Psi_P = G^\nu$ for some $\nu \geq 2$ and $G \in \overline{\mathbb{F}_p(E)}$. In particular, the function
\[
\widetilde{\Psi}_P(Q)=\Psi_P(Q) \cdot \psi^{2}_{\ell}(Q)\cdot \psi^{2}_{k}(Q+mP)
\]
has the property that each zero of $\widetilde{\Psi}_P$ has multiplicity at least two. 

Without loss of generality, let us assume that $k>\ell$. Since, as we have mentioned, the polynomial $\psi_{k}$ has $k^2-1>1$ simple roots in $\overline{\mathbb{F}_p}$, each root $Q_0$  of $\psi_{k}$, is also a root at least one of the $\psi_{k}(Q_0+mP)$, $\psi_{\ell}(Q_0)$ or $\psi_{\ell}(Q_0+mP)$ is zero.

Of course, we have $k^2-\ell^2\geq 3$. Therefore, we can choose $Q_0$ so that $\psi_{\ell}(Q_0) \neq 0$. Now, if either of $\psi_{k}(Q_0+mP)= 0$ or $\psi_{\ell}(Q_0+mP)= 0$ is zero and thus $k \ell m P = 0$.  However, since $\ord P\not\mid m$  this contradicts our assumption $\gcd(k\ell,\ord P)=1$  and 
concludes the proof. 
\end{proof}
\subsection{Some character sums with division polynomials} 

We now  derive from~\cite[Lemma~4.2]{ShSt}. 

\begin{lem}\label{lem:key}
Let $\chi$ be any non-principal multiplicative character on $\F_p^{*}$, $R$ be as in~\eqref{eqn:R}, and $\Psi\in \F_p(E)$ be of degree $\deg \Psi$. Assume that $\Psi$ is not a non-trivial power of a function in $\overline{\F_p(E)}$. Then, we have the following estimate for any integer $a$
$$\sum_{n=1}^R \chi(\Psi(nP))\er(an)\ll  \deg \Psi \cdot \sqrt{p}.$$
\end{lem} 

\begin{proof}
    Since $R=dr$, where $r=\ord P$, we can write
    \begin{align*}
    \sum_{n=1}^R \chi(\Psi(nP))\er(an) & =\sum_{k = 0}^{d-1} \sum_{n=1}^r \chi(\Psi((nd-k)P))\er(a(nd-k))\\
    &=\sum_{k = 0}^{d-1} \er(-ak) \sum_{n=1}^r \chi(\Psi((nd-k)P))\esr(an).
    \end{align*}
    The result now follows by applying~\cite[Lemma~4.2]{ShSt} to each of the inner sums over $n$.  
\end{proof}

Consequently, we deduce the following estimate.

\begin{lem}\label{lem:completesumd}
Assume that $R\ge p^{1/2}\exp\(2.1\log p/\log \log p\)$. For any integer $a$, we have
    $$ \sum_{n=1}^R \chi (\psi_n(P)) \er(a n)\ll 
    p^{1/12}R^{5/6}(\log \log R)^{1/3}.$$
\end{lem}

\begin{proof}
   The proof proceeds in the same manner as that of~\cite[Theorem~5.1]{ShSt}; for completeness, we briefly sketch the argument. The only modification involves the choice of the set of integers
$$
 \cR_d = \left\{1\le r \le L:~r \equiv 1 \bmod {d},\  \gcd(r,R) = 1\right\},
$$
for a suitably chosen parameter $L$. 
Let $Q$ be the largest factor of $R$ with $\gcd(d, Q) = 1$. 
Clearly for $r \equiv 1 \bmod {d}$  the condition $\gcd(r,R) = 1$ is equivalent to 
$\gcd(r,Q) = 1$. Writing $r = du +1$, a simple inclusion--exclusion argument gives
$$
\sharp \cR_d  =  \sum_{e\mid Q} \mu(e) \sum_{\substack{0 \le u \le (L-1)/d\\ du +1 \equiv 0 \bmod e}} 1
 =   \sum_{e\mid Q} \mu(e) \(\frac{L-1}{de} + \vartheta_e\)
$$
with some $\vartheta_e \in [-1,1]$, $e \mid Q$. Hence, for  some $\vartheta \in [-1,1]$, we have 
\begin{align*}
\sharp \cR_d  &=  \sum_{e\mid Q} \mu(e) \sum_{\substack{0 \le u \le (L-1)/d\\ du +1 \equiv 0 \bmod e}} 1
 =   \sum_{e\mid Q} \mu(e) \(\frac{L-1}{de} + \vartheta_e\)\\
& = \frac{L-1}{d}    \sum_{e\mid Q} \frac{ \mu(e)}{e} +  \vartheta \tau(Q) 
=    \frac{L-1}{d}    \prod_{\substack{\ell \mid Q\\\ell~\text{prime}}} \(1 - \frac{1}{\ell}\)+   \vartheta \tau(Q)  \\
&  =    \frac{\varphi(Q)(L-1)}{dQ}+   \vartheta \tau(Q)  , 
    \end{align*} 
see~\cite[Theorem~62]{HaWr}.

Thus, using the elementary observations 
$$
 \frac{\varphi(Q)}{Q} \ge  \frac{\varphi(R)}{R}
\mand \tau(Q) \le \tau(R)
$$
and  the following well-known estimates
$$
\tau(R) \leq \exp\(\(\log 2 + o(1)\) \log R/\log \log R\)
 \quad \text{and}\quad    \varphi(R) \gg \frac{R}{\log \log R},
$$
as $R \to \infty$,
see~\cite[Theorems~317 and~328]{HaWr}, we see that 
$$
\sharp\,\cR_d  \ge  \frac{\varphi(R)(L-1)}{dR}- \tau(R) \gg \frac{L}{\log \log R}
$$ 
provided that 
 \begin{equation}\label{eq:Lchoose}
L \ge \exp\(0.7 \log p/\log \log p\). 
\end{equation}

The desired sum is then bounded by $|W|$, where
$$
W = \frac{1}{\sharp\,\cR_d} \sum_{r\in \cR_d} \sum_{n=1}^R \chi(\psi_{n r}(P))\, \er(anr).
$$ 
We interchange the sums and apply the Cauchy--Schwarz inequality, as exactly done in~\cite[Theorem~5.1]{ShSt}. Thanks to Lemma~\ref{lem:mult}, we have the following identity for any $r \in \cR_d$
$$\chi(\psi_{nr}(P))=\chi(\psi_n(P))\chi(\psi_{r}(nP)).$$
Applying the identity, we have
\begin{align*}
    W^2 &\leq   \frac{R}{\(\sharp\,\cR_d\)^2} \sum_{n=1}^R\left|\sum_{r \in \cR_d}\chi(\psi_{r}(nP))\, \er(anr)\right|^2 \\
    &\leq   \frac{R}{\(\sharp\,\cR_d\)^2} \sum_{r_1, r_2 \in \cR_d}\left|\ \sum_{n=1}^R \chi(\varPsi_{r_1,r_2}(nP)) \er(a(r_1 - r_2)n)\right|,
    \end{align*} 
where $\varPsi_{r_1,r_2}\in \F_p(E)$ is given by the following
$$\varPsi_{r_1,r_2}(Q)= \psi_{r_1}(Q)\cdot \psi_{r_2}^{-1}(Q).$$

For any $r_1\neq r_2$, it is clear that $\varPsi_{r_1,r_2}$ is not power of a function in $\overline{\F_p(E)}$
(since zeros of $\psi_{r}(Q)$ in the function field of $E$ are exactly distinct  torsion points of order dividing $r$, see also the proof of Lemma~\ref{lem:notpower}). 

Estimating the contributions from the integers $r_1=r_2\in \cR_d$  trivially, and applying Lemma~\ref{lem:key} for each such $r_1\neq r_2\in \cR_d$,  we have
\begin{align*}
 W&\ll R\(\sharp\,\cR_d\)^{-1/2}+ R^{1/2}Lp^{1/4}\\
&\ll   
R L^{-1/2}(\log \log R)^{1/2}+ R^{1/2}Lp^{1/4}.
\end{align*}
The result follows by choosing $L=  
R^{1/3}p^{-1/6}(\log \log p)^{1/3}$, as clearly it satisfies~\eqref{eq:Lchoose} due to the assumption on $R$.
\end{proof}

\section{Bounds of some double character sums with division polynomials} 

\subsection{Small $H$}
We denote $\ochi$ to be the complex conjugate  character of a multiplicative character $\chi$ of $\mathbb{F}^{*}_p$, that is, $\chi\ochi=1$. To prove Theorem~\ref{thm:intshift}, we first need the following estimate regarding the correlation between $\chi$ and $\ochi$.

Denote 
\[
T_{\chi, P} (H)   =   
\sum_{h=1}^H\left|\sum_{n=1}^R \chi(\psi_n(P))\ochi(\psi_{n+h}(P))\right|^2.
\]

\begin{lem}\label{lem:smallh}  
 Let $R$ be as in~\eqref{eqn:R}. For any integer $1\leq H\leq R$, we have
\[
T_{\chi, P} (H)  \ll  
H^{1/2}R^2  \exp\(O\((\log R)^{1/2}/\log \log R\)\)+ H^3R p^{1/2}.
\] 
\end{lem}
    
\begin{proof} Let  $\cE$ be the set of positive integers $e\le H$ which factors into primes which are also prime divisors of $R$.
For each $e\in \cE$ we  consider the following set of integers
$$
\cK_e = \{1\le k \le H/e:~\gcd(k,R) = 1\}.
$$
Clearly we have the following  partition
$$
\{1, \ldots, H\} = \bigsqcup_{e \in \cE} \(e\cK_e\) 
$$
in disjointed sets $e\cK_e=  \{ek:~k \in \cK_e\}$. Hence
\begin{equation}\label{eq:T Se}
T_{\chi, P} (H,N)  \le  \sum_{e \in \cE} S_e,
\end{equation}
where 
$$
S_e =\sum_{k\in \cK_e}\left|\sum_{n=1}^R \chi(\psi_n(P))\ochi(\psi_{n+ek}(P))\right|^2.
$$
Clearly, we can replace $n$ in the inner sums with $kn$ for any $k \in \cK_e$, hence
$$
S_e =\sum_{k\in \cK_e}\left|\sum_{n=1}^R \chi(\psi_{kn}(P))\ochi(\psi_{k(n+e)}(P))\right|^2.
$$

For any $k\in \cK_e$,  by Lemma~\ref{lem:mult} we have the following identity
\begin{align*}
&\chi(\psi_{kn}(P))\ochi(\psi_{k(n+e)}(P))\\
&\qquad=\chi\(\psi_{k}\(nP\) \psi_{k}^{-1}\((n+e)P\)\)\chi\(\psi_n\(P\) \psi^{-1}_{n+e}\(P\)\)^{2k^2},
\end{align*}
with the convention that if $ \psi_{k}\((n+e)P\) = 0$ or $\psi_{n+e}\(P\)=0$ then the whole expression 
is equal to zero. We follow a similar convention below as well. 

Therefore, 
\begin{equation}\label{eq:SeWe}
S_e\le   \sum_{j=1}^dW_{j,e}, 
\end{equation}
with 
\begin{align*}
W_{j,e}  & =  \sum_{\substack{k\in \cK_e\\ k\equiv j \bmod  d}}
\biggl|\sum_{n=1}^R
\chi\(\psi_{k}\(nP\) \psi_{k}^{-1}\((n+e)P\)\) \\
& \qquad \qquad  \qquad \qquad  \qquad \qquad \chi\(\psi_n(P)  \psi^{-1}_{n+e}\(P\)\)^{2k^2}
\biggr|^2\\
&= \sum_{m,n =1}^R
\chi\(\psi_m(P) \cdot \psi^{-1}_{m+e}(P) \cdot
\psi^{-1}_{n}(P) \cdot \psi_{n+e}(P)\)^{2j^2}\\
&  \qquad \qquad  \qquad \qquad  \qquad \qquad  \qquad \qquad 
\sum_{\substack{k\in \cK_e\\ k\equiv j \bmod  d}} \varXi_{k,m,n},
\end{align*} 
where we denote
\[
\varXi_{k,m,n}=\chi\(\psi_{k}\(mP\) \psi_{k}^{-1}\((m+e)P\)\)  \ochi\(\psi_{k}\(nP\) \psi_{k}^{-1}\((n+e)P\)\) .
\]

By the Cauchy--Schwarz inequality, and changing the order of summation, we see that 
\begin{align*}
W_{j,e}^2 & \leq  R^2 \sum_{m,n =1}^R
      \left|   \sum_{\substack{k\in \cK_e\\ k\equiv j \bmod  d}}  \varXi_{k,m,n} \right|^2\\
      & = R^2\sum_{\substack{k,\ell \in \cK_e\\ k\equiv\ell\equiv j \bmod  d}} 
    \sum_{m,n =1}^R \varXi_{k,m,n}\overline{ \varXi_{\ell,m,n}}. 
\end{align*}
Denoting
\[\varPsi_{k,\ell}(Q)=\psi_{k}(Q)\cdot \psi_{k}^{-1}(Q+eP)\cdot \psi_{\ell}^{-1}(Q)\cdot \psi_{\ell}(Q+eP),
\]
we see that 
\begin{align*}
W_{j,e}^2&\leq R^2  \sum_{\substack{k,\ell \in \cK_e\\ k\equiv\ell\equiv j \bmod  d}}  
      \sum_{m,n =1}^R
      \chi\(\varPsi_{k,\ell}(nP)\)   \overline{\chi\(\varPsi_{k,\ell}(mP)\)}\\
    &  =   R^2\sum_{\substack{k,\ell \in \cK_e\\ k\equiv\ell\equiv j \bmod  d}}    \left| \sum_{n=1}^R
      \chi(\varPsi_{k,\ell}(nP))\right|^2. 
\end{align*}

By Lemma~\ref{lem:notpower}  we have 
\[
\deg  \varPsi_{k,\ell} \ll k^2+\ell^2\ll (H/e)^2.
\]

In particular, using Lemma~\ref{lem:key} for any distinct integers $k, \ell \in \cK_e$ and 
 estimating the contribution from $k= \ell$  trivially, since the implied constants 
 may depend on $d$, we obtain 
\[
W_{j,e}^2 \ll R^2\( \sharp\, \cK_e R^2 + \( \sharp\, \cK_e\) ^2  \(\(He^{-1}\)^2 p^{1/2} \)^2 \).
 \] 
 
 Since $\sharp\, \cK_e \le H/e$, we see that 
\[
W_{j,e} \ll   H^{1/2}R^2 e^{-1/2} +  \ H^3Rp^{1/2}e^{-3}. 
\]
Substituting this bound in~\eqref{eq:SeWe}, we derive 
\begin{align*}
S_e & \ll  H^{1/2}R^2 e^{-1/2} +   H^3Rp^{1/2}e^{-3}.
\end{align*}

Recalling~\eqref{eq:T Se}, we see that 
\begin{equation}\label{eq:T rho}
\begin{split}
 T_{\chi, P} (H)  & \ll   \sum_{e \in \cE}  \(H^{1/2}R^2 e^{-1/2} +H^3R p^{1/2} e^{-3}\)\\
 & \ll H^{1/2}R^2 \sum_{e \in \cE} e^{-1/2} + H^3R  p^{1/2} \sum_{e \in \cE}  e^{-3}\\
  & \ll H^{1/2}R^2\rho + H^3R p^{1/2}, 
\end{split}
\end{equation}
where
 \[
 \rho = \sum_{e \in \cE} e^{-1/2}.
 \]
 Let $\nu = \omega(R)$ and let $p_i$ denote the $i$th prime. Then 
\begin{equation}\label{eq:rho}
\begin{split}
 \rho & \le \prod_{p \mid R}  \sum_{j=1}^\infty p^{-j/2} \le 
 \prod_{i=1}^\nu  \sum_{j=1}^\infty p_i^{-j/2} \\
 & = \prod_{i=1}^\nu  \(1 + p_i^{-1/2} + O\(p_i^{-1}\)\) = \exp \lambda , 
\end{split}
\end{equation}
  where 
\begin{equation}\label{eq:lambda}
\begin{split}
  \lambda & = \sum_{i=1}^\nu  \log \(1 + p_i^{-1/2} + O(p_i^{-1})\)
  \ll \sum_{i=1}^\nu  p_i^{-1/2} \\
  & \ll \sum_{i=1}^\nu \(i \log i\)^{-1/2} \ll \nu^{1/2}  \(\log \nu\)^{-1/2}. 
\end{split}
\end{equation}
Since obviously $\nu! \le R$, from the Stirling formula, we see that
\[
\nu \ll \frac{\log R}{\log \log R}
\]
and thus  $\rho \le \exp\(O\((\log R)^{1/2}/\log \log R\)\)$. Substituting this bound on $\rho$ in~\eqref{eq:T rho}, we conclude the proof.  
\end{proof}

\begin{rem}\label{rem:withgamma} 
Using the so-called Rankin method, that is, using that for any $\gamma> 0$, similarly to~\eqref{eq:rho}
and~\eqref{eq:lambda}, we have  
\begin{align*}
 \rho = \sum_{e \in \cE} e^{-1/2} &\le   \sum_{\substack{e \in \N\\ p \mid e\, \Rightarrow \, p\mid R}} e^{-1/2}\(\frac{H}{e}\)^\gamma \ll H^\gamma \sum_{i=1}^\nu \(i \log i\)^{-1/2-\gamma} \\
  &= H^\gamma \nu^{1/2-\gamma}  \(\log \nu\)^{-1/2-\gamma}\\
   &= H^\gamma \exp\(O\((\log R)^{1/2-\gamma} /\log \log R\)\), 
\end{align*}
and optimising the choice of $\gamma$ one can obtain a slightly more precise bound when $H$ is 
small.
\end{rem}

 \subsection{Large $H$}\label{sec:largeh}
In this section, we obtain a similar estimate for the correlation between $\chi$ and $\ochi$ as in Lemma~\ref{lem:smallh}, for a large $H$.

We first estimate $T_{\chi,P}(H)$ for $H=R$. 

\begin{lem}\label{lem:largeh}
Assume that $R \ge p^{1/2}\exp\(2.1 \log p/\log \log p\)$. Then, we have
\[
T_{\chi,P}(R)\ll R^{8/3}p^{1/6}(\log \log R)^{2/3}. 
 \]
\end{lem}

\begin{proof}   
Using that $\chi(\psi_k(P))$, $k =1,2, \ldots$,  is  periodic with period $R$, we can derive
\begin{align*}
&T_{\chi,P}(R) \\
& \quad =\sum_{h=1}^R\left|\sum_{n=1}^R \chi(\psi_n(P))\ochi(\psi_{n + h}(P)) \right|^2\\
&\quad=\sum_{m,n=1}^R \chi(\psi_m(P)) \ochi(\psi_{n}(P)) \sum_{h=1}^R\chi(\psi_{m +h}(P)) \ochi(\psi_{n + h}(P))\\
&\quad = 
\sum_{\substack{m_1,m_2,n_1,n_2=1\\m_1-m_2 \equiv n_1 - n_2 \bmod  R}}^R 
\chi(\psi_{m_1}(P)) \ochi(\psi_{n_1}(P)
\chi(\psi_{m_2}(P)) \ochi(\psi_{n_2}(P)).
\end{align*}

Using the orthogonality of exponential functions, we write
\begin{align*}
&T_{\chi,P}(R) \\
&\quad =\sum_{m_1,m_2,n_1,n_2=1}^R 
\chi(\psi_{m_1}(P)) \ochi(\psi_{n_1}(P)
\chi(\psi_{m_2}(P)) \ochi(\psi_{n_2}(P))\\
&\qquad\qquad\qquad\qquad\qquad\qquad\qquad\cdot \frac{1}{R} \sum_{\lambda=1}^R \er\(\lambda\(m_1-m_2- n_1 +n_2\)\)\\
&\quad = \frac{1}{R} \sum_{\lambda=1}^R \left| \sum_{n=1}^R \chi (\psi_n(P)) \er(\lambda n)\right|^2  \left| \sum_{n=1}^R \chi (\psi_n(P)) \er(-\lambda n)\right|^2. 
\end{align*}
Applying  Lemma~\ref{lem:completesumd}   and the orthogonality of exponential functions to the 
remaining sum
we see that 
\begin{align*}
\sum_{h=1}^R&\left|\sum_{n=1}^R \chi(\psi_n(P))\ochi(\psi_{n + h}(P)) \right|^2\\
 & \ll \(R^{5/6}p^{1/12}(\log \log R)^{1/3}\)^2 \frac{1}{R} \sum_{\lambda=1}^R \left| \sum_{n=1}^R \chi (\psi_n(P)) \er(\lambda n)\right|^2\\
 & \ll  \(R^{5/6}p^{1/12}(\log \log R)^{1/3}\)^2 R  = R^{8/3}p^{1/6}(\log \log R)^{2/3},
   \end{align*}
which concludes the proof.
\end{proof}

\section{Proof of Theorem~\ref{thm:intshift}} 
\subsection{Preliminaries} 
We start with the following bound on weighted sums. 

\begin{lem}\label{lem:twistby1bdd}
Let $R$ be as in~\eqref{eqn:R}, and let $\alpha_n$, $n=1, \ldots, R$, be a sequence of weights with $|\alpha_n| \le 1$.  Then, we have
\[
\frac{1}{H}\sum_{h=1}^H \left|  \sum_{n=1}^R \alpha_n\chi(\psi_{n+h}(P))\right|
\ll  H^{-1/4}\left(R^{1/2}\left|T_{\chi,P}(H)\right|^{1/4}+ R\right).
\]
\end{lem}

\begin{proof}
Denote 
$$S=\sum_{h=1}^H \left|\sum_{n=1}^R \alpha_n\chi(\psi_{n + h}(P)) \right|.$$
By the Cauchy--Schwarz inequality,
\begin{align*}
|S|^2&\leq H \sum_{h=1}^H \left|\sum_{n=1}^R \alpha_n\chi(\psi_{n + h}(P))\right|^2\\
&\leq H \sum_{m,n=1}^R \overline{\alpha_m} \alpha_n \sum_{h=1}^H\chi(\psi_{m +h}(P)) \ochi(\psi_{n + h}(P)).
\end{align*}
In particular, we have
$$|S|^2\leq H \sum_{m,n=1}^R \left|\sum_{h=1}^H\chi(\psi_{m  + h}(P)) \ochi(\psi_{n+ h}(P))\right|.$$
Applying the Cauchy--Schwarz inequality again, we get
\begin{align*}
 |S|^4&\leq H^2R^2 \sum_{m,n=1}^R\,\left|\sum_{h=1}^H \chi(\psi_{m+ h}(P)) \ochi(\psi_{n+ h}(P))\right|^2\\
 &= H^2R^2 \sum_{m,n=1}^R\,\sum_{h,k=1}^H \chi(\psi_{m + h}(P)) \ochi(\psi_{n + h}(P))\\ 
&\qquad\qquad\qquad\qquad\qquad\qquad \cdot \ochi(\psi_{m + k}(P)) \chi(\psi_{n+ k}(P)).
\end{align*}

By periodicity,  we now have 
\begin{align*}
 |S|^4&\leq H^3R^2     \sum_{|h| \leq H}  \left|
 \sum_{m,n=1}^R  \chi(\psi_{m}(P)) \ochi(\psi_{n}(P))  \ochi(\psi_{m + h}(P)) \chi(\psi_{n+ h}(P)) \right|\\
 &= H^3R^2     \sum_{|h| \leq H}  \left|
\sum_{n=1}^R  \chi(\psi_{n}(P))  ) \ochi(\psi_{n+ h}(P)) \right|^2\\
 & = 2 H^3R^2   \sum_{h=1}^H  \left|
 \sum_{n=1}^R \chi(\psi_{n}(P))  ) \ochi(\psi_{n+ h}(P)) \right|^2 + H^3R^4\\
 &=  2 H^3R^2 T_{\chi,P}(H)+ H^3R^4, 
\end{align*} 
and the result follows. 
\end{proof}


Let $a$ be an arbitrary integer. Applying Lemma~\ref{lem:twistby1bdd}, we have
\begin{align*}
    &\frac{1}{H^{m-1}} \sum_{h_2, \ldots, h_m=1}^H    \left| \sum_{n=1}^R \chi(\psi_n(P)) 
    \prod_{j=2}^m \chi(\psi_{n+h_j}(P))\, \er(an) \right| \\
    &\qquad\qquad\qquad\qquad\qquad\qquad\ll  H^{-1/4}\left(R^{1/2}\left|T_{\chi,P}(H)\right|^{1/4}+ R\right).
\end{align*}

On the other hand, by the periodicity and Lemma~\ref{lem:twistby1bdd}, we have
\begin{align*}
&\frac{1}{H^{m}} \sum_{h_1, \ldots, h_m=1}^H
    \left| \sum_{n=1}^R 
    \chi(\psi_{n+h_1}(P))\prod_{j=2}^m \chi(\psi_{n+h_j}(P)) \er(an)\right| \\
    &\qquad\qquad =\frac{1}{H^m} \sum_{h_1, \ldots, h_m=1}^H\left| \sum_{n=1}^R 
    \chi(\psi_{n}(P))\prod_{j=2}^m \chi(\psi_{n+h_j-h_1}(P)) \er(an)\right|\\
    &\qquad\qquad\qquad\qquad\qquad\ll  H^{-1/4}\left(R^{1/2}\left|T_{\chi,P}(H)\right|^{1/4}+ R\right).
\end{align*}

By the completing technique, see~\cite[Lemma~12.1]{IwKow}, we have
$$
U_{m, \chi, P} (H,N), V_{m, \chi, P} (H,N) 
\ll  H^{-1/4}\left(R^{1/2}\left|T_{\chi,P}(H)\right|^{1/4}+ R\right)\log R.
$$

 \subsection{Concluding the proof}
Now,  applying  Lemma~\ref{lem:smallh}, we have
\begin{align*}
U_{m,\chi,P}(H,N)&, V_{m,\chi,P}(H,N)\\
& \ll  H^{-1/8}R\exp\(O\((\log R)^{1/2}/\log \log R\)\)\\
&\qquad \qquad \quad+H^{1/2}R^{3/4}p^{1/8} \log R+H^{-1/4}R \log R. 
\end{align*}
Clearly the first term always dominates the third, which can be omitted and thus we obtain~\eqref{eq:B1}.

Next, assume that $R\ge p^{1/2}\exp\(2.1\log p/\log \log p\)$. Hence we can apply Lemma~\ref{lem:largeh}, and derive 
\begin{align*}
& U_{m,\chi,P}(H,N), V_{m,\chi,P}(H,N)\\
&\qquad \qquad  \ll H^{-1/4} R^{7/6}p^{1/24} (\log R) (\log \log R)^{1/6}  + H^{-1/4}R \log R.
\end{align*}
Again,  the first term always dominates  and  we obtain~\eqref{eq:B2}, 
which completes the proof.

\section*{Acknowledgement}

During the preparation of this work, the  authors   were partially supported by the
Australian Research Council Grant DP230100530.

\end{document}